\documentclass{siamltex}
\usepackage{amsfonts,amsmath,mathrsfs}
\usepackage{amssymb}
\usepackage{hyperref}
\usepackage{graphicx}
\usepackage{wrapfig}
\usepackage{subfig}
\usepackage{epsfig}
\usepackage{float}
\usepackage{bbm}
\usepackage{booktabs}

\newcommand{\T}{{\rm T}}

\newcommand{\R}{\mathbb{R}}

\newcommand{\norm}[1]{\left\| #1 \right\|}
\newtheorem{remark}[theorem]{Remark}
\newtheorem{algorithm}{Algorithm}

\newcommand{\eps}{\varepsilon}
\newcommand{\id}{\mathbb{I}}

\title{Compensation for geometric modeling errors by electrode movement in electrical impedance tomography}

\author{N. Hyv\"onen\footnotemark[2]
\and H. Majander\footnotemark[3]
\and S. Staboulis\footnotemark[4]
}

\begin{document}
\maketitle

\renewcommand{\thefootnote}{\fnsymbol{footnote}}

\footnotetext[2]{Aalto University, Department of Mathematics and Systems Analysis, P.O. Box 11100, FI-00076 Aalto, Finland (nuutti.hyvonen@aalto.fi). The work of NH was supported by the Academy of Finland (decision 267789).}

\footnotetext[3]{Aalto University, Department of Mathematics and Systems Analysis, P.O. Box 11100, FI-00076 Aalto, Finland and Ecole
Polytechnique, Centre de Math\'ematiques Appliqu\'es, Route de Saclay,
91128 Palaiseau Cedex, France (helle.majander@aalto.fi). The work of HM was supported by the Foundation for Aalto University Science and Technology and by the Academy of Finland (decision 267789).}

\footnotetext[4]{Technical University of Denmark, Department of Applied Mathematics and Computer Science, Asmussens Alle, Building 322, DK-2800, Kgs. Lyngby, Denmark (ssta@dtu.dk).}

\begin{abstract}
Electrical impedance tomography aims at reconstructing the conductivity inside a physical body from boundary measurements of current and voltage at a finite number of contact electrodes. In many practical applications, the shape of the imaged object is subject to considerable uncertainties that render reconstructing the internal conductivity impossible if they are not taken into account. This work numerically demonstrates that one can compensate for inaccurate modeling of the object boundary in two spatial dimensions by estimating the locations and sizes of the electrodes as a part of a reconstruction algorithm. The numerical studies, which are based on both simulated and experimental data, are complemented by proving that the employed {\em complete electrode model} is approximately conformally invariant, which suggests that the obtained reconstructions in mismodeled domains reflect conformal images of the true targets. The numerical experiments also confirm that a similar approach does not, in general, lead to a functional algorithm in three dimensions.
\end{abstract}

\renewcommand{\thefootnote}{\arabic{footnote}}

\begin{keywords}
Electrical impedance tomography, geometric modeling errors, electrode movement, inaccurate measurement model, complete electrode model, conformal invariance
\end{keywords}

\begin{AMS}
65N21, 35R30
\end{AMS}

\pagestyle{myheadings}
\thispagestyle{plain}
\markboth{N. HYV\"ONEN, H. MAJANDER, AND S. STABOULIS}{ELECTRODE MOVEMENT AND MODELING ERRORS IN EIT}

\section{Introduction}
\label{sec:introduction}
{\em Electrical impedance tomography} (EIT) aims at reconstructing the conductivity (or admittivity) inside a physical body from boundary measurements of current and voltage at a finite number of contact electrodes. The most accurate way to model the function of an EIT device is employing the {\em complete electrode model} (CEM), which takes into account the electrode shapes and the contact resistances (or impedances) caused by resistive layers at the electrode-object interfaces~\cite{Cheng89,Somersalo92}. For information on potential applications of EIT, we refer to the review articles~\cite{Borcea02,Cheney99,Uhlmann09} and the references therein. 

In a real-world setting for EIT, the conductivity is almost never the only unknown: the information on the contact resistances, the positioning of the electrodes and the shape of the imaged object is typically incomplete as well. As an example, when imaging (a part of) a human body, the domain shape and the contact resistances obviously depend on the examined patient and the localization of the electrodes is prone to suffer from considerable inaccuracies. As observed already in~\cite{Barber88, Breckon88, Kolehmainen97}, even slight mismodeling of the measurement setting typically ruins the reconstruction of the conductivity, and so it is essential to develop algorithms that are robust with respect to (geometric) modeling errors.

The most straightforward way to cope with unknown contact resistances, electrode locations and boundary shape in EIT is arguably to include their estimation as a part of a (Bayesian) output least squares reconstruction algorithm. The fundamental requirement for this approach is the ability to compute/approximate the (Fr\'echet) derivatives of the electrode measurements with respect to the corresponding model parameters, which has been established in~\cite{Darde12,Darde13a,Vilhunen02}. In~\cite{Darde13a,Darde13b}, this approach of simultaneous estimation of the conductivity and the geometric parameters was successfully tested with both simulated and experimental data; the computations in \cite{Darde13a,Darde13b} were performed on three-dimensional {\em finite element} (FE) meshes, although the considered measurement settings were essentially two-dimensional,~i.e.,~homogeneous along one of the coordinate axes. However, the algorithm introduced in~\cite{Darde13a,Darde13b} carries an obvious weakness: the computation of the needed shape derivatives with respect to the object boundary suffers from numerical instability that requires the use of relatively dense FE meshes and/or artificially high contact resistances in order to regularize the CEM forward problem (cf.~\cite{Darde16}). On the other hand, the computation of the Fr\'echet derivatives with respect to the electrode positions and shapes does not suffer from as severe instability.

This work demonstrates that in two spatial dimensions one can compensate for a mismodeled object shape by including merely the estimation of the electrode locations and sizes in an output least squares reconstruction algorithm of EIT, thus circumventing the issues with the stability of shape derivatives documented in~\cite{Darde13a,Darde13b}. To be more precise, the reconstruction is formed in a simple but inaccurate model domain, but in addition to the conductivity and the contact resistances also the electrode positions and sizes are estimated by the reconstruction algorithm; see~\cite{Winkler14} for a preliminary numerical example. We justify this approach theoretically by
proving that the CEM is in a certain sense approximately conformally invariant, which suggests that the reconstruction in the model domain approximates a conformal image of the true target. (In fact, allowing spatially varying contact resistances would make the CEM fully conformally invariant~\cite{Rieder14}, but such a setting cannot be considered practical from the standpoint of reconstruction algorithms.) Unfortunately, the availability of a large family of conformal mappings also seems to be a necessary condition for the full functionality of the introduced approach: according to our numerical experiments, errors in the model for the object boundary cannot, in general, be compensated by allowing electrode movement in three spatial dimensions. For related algorithms, see \cite{Boyle12, Jehl15}.

To complete this introduction, let us present a brief survey of the previous methods for recovering from an unknown exterior boundary shape in EIT. In {\em difference} imaging, electrode measurements are performed at two time instants and the corresponding {\em change} in the conductivity is reconstructed~\cite{Barber84}. The modeling errors partly cancel out when the difference data are formed, allowing a reconstruction without substantial artifacts. However, difference imaging is approximative as its functionality has only been justified via a linearization of the forward model. Even more importantly, difference data are not always available, which is the premise of this work. The first generic algorithm for recovering from an inaccurate boundary shape in {\it absolute} EIT imaging was introduced by~\cite{Kolehmainen05,Kolehmainen07}, where the mismodeled geometry is taken into account by reconstructing a (slightly) anisotropic conductivity. The main weakness of the algorithm of~\cite{Kolehmainen05,Kolehmainen07} is the difficulty in generalizing it to three dimensions. The approximation error method~\cite{Kaipio05} was successfully applied to EIT with an inaccurately known boundary shape in~\cite{Nissinen11,Nissinen11b}: the error originating from the uncertainty in the measurement geometry is represented as a stochastic process whose second order moments are approximated in advance based on the prior probability densities for the conductivity and the boundary shape. A reconstruction of the conductivity is then formed via statistical inversion.

This text is organized as follows. Section~\ref{sec:forward} recalls the CEM and considers its differentiability with respect to different model parameters. In Section~\ref{sec:cinvariance}, we demonstrate that the CEM is approximately conformally invariant in two dimensions. The reconstruction algorithm, which aims at computing a {\em maximum a posteriori} (MAP) estimate for the conductivity and other unknown parameters within the Bayesian paradigm, is introduced in Section~\ref{sec:inverse}. The numerical tests are described in Section~\ref{sec:numerics}; both simulated and experimental data are considered. Finally, Section~\ref{sec:conclusion} lists the concluding remarks.

\section{Forward model and its properties}
\label{sec:forward}

In this section, we first introduce the CEM and subsequently consider its Fr\'echet differentiability with respect to different model parameters.

\subsection{Complete electrode model}
\label{sec:CEM}

In practical EIT, $M \geq 2$ contact electrodes $\{ E_m\}_{m=1}^M$ are attached to the exterior surface of a body $\Omega \subset \R^n$, $n=2$ or $3$, which is assumed to have a connected complement. A net current $ I_m\in\R $ is driven through the corresponding electrode $ E_m $ and the resulting constant electrode potentials $ U = [U_1,\ldots,U_M]^{\rm T} \in \R^M $ are measured. 
As there are no sinks or sources inside the object, any meaningful current pattern $ I = [I_1,\ldots,I_M]^{\rm T} $ belongs to the zero-mean subspace $\R^M_\diamond \subset \R^M$. The contact resistances at the electrode-object interfaces are modeled by $z = [z_1,\dots, z_M]^{\rm T} \in \R_+^M$.

We assume that $\Omega$ is a bounded domain with a smooth boundary. Moreover, the electrodes $\{ E_m\}_{m=1}^M$ are identified with the open, nonempty subsets of $\partial \Omega$ and assumed to be mutually well-separated, i.e., $\overline{E}_k \cap \overline{E}_l = \emptyset$ for $k \not= l$. We denote $E = \cup E_m$ and assume that $\partial E$ is a smooth submanifold of $\partial \Omega$. The mathematical model that most accurately predicts real-life EIT measurements is the CEM \cite{Cheng89}, which is described by an elliptic mixed Neumann--Robin boundary value problem: the electromagnetic potential $u$ and the potentials on the electrodes $U$ satisfy
\begin{equation}
\label{eq:cemeqs}
\begin{array}{ll}
\displaystyle{\nabla \cdot(\sigma\nabla u) = 0 \qquad}  &{\rm in}\;\; \Omega, \\[6pt] 
{\displaystyle{\nu\cdot\sigma\nabla u} = 0 }\qquad &{\rm on}\;\;\partial\Omega\setminus\overline{E},\\[6pt] 
{\displaystyle u+z_m{\nu\cdot\sigma\nabla u} = U_m } \qquad &{\rm on}\;\; E_m, \quad m=1, \dots, M, \\[2pt] 
{\displaystyle \int_{E_m}\nu\cdot\sigma\nabla u\,{\rm d}S} = I_m, \qquad & m=1,\ldots,M, \\[4pt]
\end{array}
\end{equation}
interpreted in the weak sense. Here, $ \nu $ is the exterior unit normal of $ \partial\Omega $ and the symmetric conductivity $\sigma: \Omega \to \R^{n \times n}$ that characterizes the electric properties of the medium is assumed to satisfy
\begin{equation}
\label{eq:sigma}
\varsigma_- \id \leq \sigma \leq \varsigma_+ \id, \qquad \varsigma_-,  \varsigma_+ > 0, 
\end{equation}
almost everywhere in $\Omega$, with the inequalities understood in the sense of positive definiteness and $\id \in \R^{n\times n}$ being the identity matrix.

Given an input current pattern $ I\in \R^M_\diamond $ as well as
the conductivity $\sigma$ and the contact resistances $z$, 
the spatial electric potential $ u\in H^1(\Omega) $ and the electrode potentials $ U\in \R^M $ are uniquely determined by \eqref{eq:cemeqs} up to a common additive constant, i.e., up to the ground level of potential~\cite{Somersalo92}. This solution pair depends continuously on the data in $\mathcal{H}(\Omega) :=  (H^1(\Omega)\oplus \R^M)/\R $, which is here equipped with the {\em electrode-dependent} norm
\begin{equation*}
\norm{(v,V)}_{\mathcal{H}(\Omega)} = \inf_{c\in\R}\Big\{ \norm{v-c}_{H^1(\Omega)}^2 + \sum_{m=1}^M\|V_m-c \|^2_{L^2(E_m)} \Big\}^{1/2}.
\end{equation*}
To be more precise,
\begin{equation}
\label{eq:bound}
\| (u,U) \|_{\mathcal{H}(\Omega)} \leq \frac{C}{\min \{\varsigma_-, z_1^{-1}, \dots, z_M^{-1} \}}  \left( \sum_{m=1}^M |I_m|^2/|E_m| \right)^{1/2},
\end{equation}
where $|E_m|$ is the area/length of $E_m$ and $C = C(\Omega) > 0$ does not depend on $\sigma$, $z$ or the geometry of the electrodes as a subset of $\partial \Omega$ (cf.~\cite[Section~2]{Hyvonen04} and~\cite[(2.4)]{Hanke11b}).  It can be shown that the interior electromagnetic potential $u$ exhibits higher Sobolev regularity of the order $H^{2-\epsilon}(\Omega)$, $\epsilon >0$, if $\sigma$ is Lipschitz continuous in $\overline{\Omega}$, meaning that also
\begin{equation}
\label{eq:regularity}
u|_{\partial\Omega}\in H^{3/2-\epsilon}(\partial\Omega)/ \R, \qquad 
u|_{\partial E} \in H^{1-\epsilon}(\partial E)/ \R  
\end{equation}
due to the trace theorem (cf.~\cite[Remark~1]{Darde12} and \cite{Grisvard85}).

We define the measurement, or current-to-voltage, map of the CEM as
\begin{equation}\label{eq:measmat}
R: I \mapsto U, \quad \R^M_\diamond \to \R^M / \R .
\end{equation} 
Take note that we equip the quotient space $\R^M / \R$ with its {\em natural} norm
$$
\| V \|_{\R^M / \R} \, = \, \inf_{c \in \R} \| V - c \mathbf{1} \|_{\R^M}
$$
where $\mathbf{1} = [1, \dots, 1]^{\rm T} \in \R^M$. 

\subsection{Fr\'echet derivatives}
\label{sec:Frechet}
In this section, we summarize some relevant Fr\'echet differentiability results for the measurement map $R: \R^M_\diamond \to \R^M / \R$ with respect to the model parameters in~\eqref{eq:cemeqs}; for more details, see~\cite{Darde12,Darde13a,Kaipio00,Lechleiter06,Vilhunen02}. We start by perturbing $\partial E$ and introducing the corresponding shape derivative.

The measurement map of \eqref{eq:measmat} can be interpreted as a function of two variables,
\[
R: (I, a) \mapsto U(I,a), \quad \R^M_\diamond \times \mathcal{B}_d
\to \R^M / \R,
\]
where $\mathcal{B}_d \subset [C^1(\partial E)]^n$ is an origin-centered open ball of radius $d>0$. The pair $(u(I,a),U(I,a))$ is the solution of \eqref{eq:cemeqs} when the electrodes $E_m$, $m=1, \dots, M$, are replaced by the perturbed versions defined by the boundaries
\begin{equation}
\label{eq:perturbed}
\partial E_m^a = \big\{ P_x \big(x + a(x) \big) \, \big| \, x \in \partial E_m \big\} \subset \partial \Omega, \qquad m=1,\dots, M,
\end{equation}
where $P_x$ is the projection in the direction of $\nu(x)$ onto $\partial \Omega$. If $d > 0$ is chosen small enough, the above definitions are unambiguous in the sense that $\{E_m^a \}_{m=1}^M$ is a set of feasible electrodes on $\partial \Omega$~\cite{Darde12}; in what follows, we will implicitly assume that this is the case. For the proof of the following theorem, we refer to~\cite{Darde12}.

\begin{theorem}
\label{thm:reunaderivaatta}
Suppose that the conductivity $ \sigma $ belongs to $C^1(\overline{\Omega},\R^{n\times n})$. 
Then $R:  \R^M_\diamond \times \mathcal{B}_d \to \R^M / \R$ is Fr\'echet differentiable with respect to its second variable at the origin.
\end{theorem}

Recall that this means there exists a (bi)linear and bounded map $U'(I,0)$ from $[C^1(\partial E)]^n$ to $\R^M/\R$ such that 
\begin{equation}\label{eq:diffquo}
\lim_{0\not=a\to 0} \frac{1}{\norm{a}_{C^1}}\| U(I,a) - U(I,0) - U'(I,0)a \|_{\R^M / \R} = 0, \qquad a\in [C^1(\partial E)]^n,
\end{equation}
for any $I \in \R^M_\diamond$. Moreover, if $I, \hat{I}\in\R^M_\diamond $ are electrode current patterns and the pairs $(u,U), (\hat{u},\hat{U})$ are the respective solutions of \eqref{eq:cemeqs}, then $U'(I,0)a $ can be sampled via the relation~\cite{Darde12}
\begin{equation}
\label{eq:sampling}
U'(I,0)a\cdot \hat{I} = 
- \sum_{m=1}^M \frac{1}{z_m}\int_{\partial E_m}(a \cdot \nu_{\partial E_m})(U_m-u)
(\hat{U}_m-\hat{u}) \,{\rm d}s,
\end{equation}
where $\nu_{\partial E_m}$ is the exterior unit normal of $\partial E_m$ lying in the tangent bundle of $\partial\Omega$. Observe also that the integrals on the right hand side of \eqref{eq:sampling} are well-defined due to~\eqref{eq:regularity}, and they reduce to pointwise evaluations when $n=2$.

Obviously, the measurement map $R$ can also be treated as a function of four variables by writing
\[
R: 
\left\{
\begin{array}{l}
(I, \sigma, z, a) \mapsto U(I,\sigma, z, a), \\[2mm]
\mathcal{D} := \R^M_\diamond \times \Sigma \times \R_+^M \times \mathcal{B}_d \to \R^M / \R,
\end{array}
\right.
\]
where
$$
\Sigma = \big\{
\sigma \in C^1\big(\overline{\Omega}, \R^{n \times n}\big) \ \big| \ \sigma = \sigma^{\rm T} \textrm{ and satisfies } \eqref{eq:sigma} \textrm{ for some } \varsigma_-, \varsigma_+ > 0
\big\}
$$ 
is a set of plausible conductivities. The differentiability of $ U =  U(I,\sigma, z, a)$ with respect to its second argument is known even for considerably less regular conductivities (see, e.g., \cite{Kaipio00,Lechleiter06}), and that with respect to the contact resistances is straightforward to establish and has been utilized in many numerical algorithms (cf., e.g., \cite{Vilhunen02}). We collect the needed differentiability results in the following corollary.  
\begin{corollary}
\label{corollary}
Under the above assumptions, the measurement map of the CEM, 
\[ 
R:   \mathcal{D} \to \R^M/\R,
\]
is Fr\'echet differentiable in the set $\R^M_\diamond\times \Sigma \times \R_+^M \times\{0\} \subset \mathcal{D} $.
\end{corollary}

\begin{proof}
The assertion is a weaker version of \cite[Corollary~2.2]{Darde13b}. \qquad
\end{proof}

The numerical approximation of the (partial) Fr\'echet derivatives of $R$ with respect to $\sigma$ and $z$ has been considered in many previous works \cite{Kaipio00,Lechleiter06,Vilhunen02}, and we will compute the needed derivatives with respect to the electrode positions and shapes with the help of \eqref{eq:sampling} \cite{Darde12,Darde13a}.

\section{Approximate conformal invariance in two dimensions}
\label{sec:cinvariance}

In this section, we assume exclusively that $n=2$, denote a smooth, simply connected reference domain by $D$,\footnote{$D$ can be,~e.g.,~the open unit disk.} let $\Phi$ be a (fixed) conformal map sending $\Omega$ onto $D$, and denote its inverse by $\Psi$. As $\partial \Omega$ is also assumed to be smooth, the derivatives of $\Phi$ and $\Psi$ up to an arbitrary order are bounded on $\overline{\Omega}$ and $\overline{D}$, respectively, and $\Phi|_{\partial \Omega}$ defines a $C^\infty$-diffeomorphism of $\partial \Omega$ onto $\partial D$~\cite{Pommerenke92}. For simplicity, it is also assumed that $\sigma \in C^\infty(\overline{\Omega}, \R^{n \times n})$. 

As the general aim of this section is to study asymptotics as the electrode diameters tend to zero, in what follows we emphasize the dependence on a `width parameter' $0<h\leq1$ by denoting the electrodes and the solution to \eqref{eq:cemeqs} by $\{ E^h_m\}_{m=1}^M$ and $(u^h, U^h) \in \mathcal{H}(\Omega)$,\footnote{Observe that $\mathcal{H}(\Omega)$ also depends on $0<h\leq 1$ via its norm.} respectively. The center point (with respect to curve length) of $E^h_m$ is denoted by $y_m \in \partial \Omega$, $m=1, \dots , M$; in particular, $y_m$ is assumed to remain the midpoint of $E_m^h$ independently of $0<h\leq1$.
Moreover, the electrode widths are assumed to scale according to
\begin{equation}
\label{Eh}
|E_m^h| = h |E_m^1|, \qquad m = 1, \dots, M,
\end{equation}
where $\{ E_m^1 \}_{m=1}^M$ is a set of feasible reference electrodes with $\{ y_m \}_{m=1}^M$ as their respective centers.

Let us introduce electrodes on $\partial D$ via $\tilde{E}_m^h = \Phi(E_m^h)$, $m=1, \dots, M$, and define the `push-forward' conductivity and contact resistances for $D$ as
$$
\tilde{\sigma} = J_{\Psi}^{-1} (\sigma \circ \Psi) \big( J_{\Psi}^{-1} \big)^{\rm T} \det J_\Psi \qquad {\rm and} \qquad \tilde{z}_m = |\Phi' (y_m)|\, z_m, \quad m=1, \dots, M,
$$
respectively. Here, $J_{\Psi}: D \to \R^{2\times 2}$ is the Jacobian of $\Psi$ and $| \Phi'| = \sqrt{\det J_\Phi}$ denotes the absolute value of the (complex) derivative $\Phi'$. It is easy to see that $\tilde{\sigma}$ is a feasible conductivity, i.e., it satisfies a condition of the type \eqref{eq:sigma} almost everywhere in $D$, and that $\tilde{\sigma} = \sigma \circ \Psi$ if $\sigma$ is isotropic since $(J_\Psi^{\rm T} J_\Psi) = (\det J_\Psi) \mathbb{I}$ due to conformality. We denote by $(\tilde{u}^h, \tilde{U}^h) \in \mathcal{H}(D)$ the unique solution of the corresponding CEM problem in $D$, that is, 
\begin{equation}
\label{eq:cemeqsD}
\begin{array}{ll}
\displaystyle{\nabla \cdot (\tilde{\sigma}\nabla \tilde{u}^h) = 0 \qquad}  &{\rm in}\;\; D, \\[6pt] 
{\displaystyle{\nu\cdot\tilde{\sigma}\nabla \tilde{u}^h} = 0 }\qquad &{\rm on}\;\;\partial D \setminus\overline{\tilde{E}^h},\\[6pt] 
{\displaystyle \tilde{u}^h+\tilde{z}_m{\nu\cdot\tilde{\sigma}\nabla \tilde{u}^h} = \tilde{U}_m^h } \qquad &{\rm on}\;\; \tilde{E}_m^h, \quad m=1, \dots, M, \\[2pt] 
{\displaystyle \int_{\tilde{E}_m^h}\nu\cdot\tilde{\sigma}\nabla \tilde{u}^h\,{\rm d}S} = I_m, \qquad & m=1,\ldots,M, \\[4pt]
\end{array}
\end{equation}
where $\nu$ now denotes the exterior unit normal of $\partial D$. 

It is easy to deduce (cf.,~e.g.,~\cite{Rieder14}) that the pair $(\tilde{u}^h\circ \Phi, \tilde{U}^h)\in \mathcal{H}(\Omega)$ satisfies the first, second and fourth equations in \eqref{eq:cemeqs} --- with each $E_m$ replaced by $E^h_m$ --- but the third one transforms into the form
\begin{equation}
\label{uusikolmas}
\tilde{u}^h\circ\Phi + \frac{\tilde{z}_m}{|\Phi'|}\, \nu \cdot \sigma \nabla (\tilde{u}^h\circ \Phi) \, = \, \tilde{U}_m^h \qquad {\rm on} \ E_m^h, \qquad m=1, \dots, M,
\end{equation}
which, actually, motivates our definition of $\tilde{z}_m$. As interpreted in \cite{Rieder14}, the pair $(\tilde{u}^h\circ \Phi, \tilde{U}^h)$ satisfies a CEM forward problem in $\Omega$, but with {\em spatially varying} contact resistances; conversely, the same applies to $(u^h\circ \Psi, U^h)$ in $D$. Such a setting can be handled theoretically (cf.~\cite{Hyvonen04}), but it does not provide a reasonable computational framework for tackling the inverse problem of EIT because parametrizing and reconstructing spatially varying contact resistances needlessly complicates numerical algorithms.

Before moving on to the actual (approximative) conformal invariance result for the CEM with {\em constant} contact resistances (cf.~\cite{Rieder14}), let us generalizes/modify \cite[Lemma~3.2]{Hanke11b} so that it serves our purposes. To this end, define
\begin{equation}
\label{affa}
f \, = \, \sum_{m=1}^M I_m \, \delta_{y_m} \quad {\rm on} \ \partial \Omega,
\end{equation}
where $\delta_y \in H^{-1/2-\epsilon}(\partial \Omega)$, $\epsilon > 0$, denotes the delta distribution supported at $y \in \partial \Omega$.

\begin{lemma}
\label{tokalemma}
For any $\epsilon >0$, it holds that 
\begin{equation}
\label{Neumann}
\big\|\nu \cdot \sigma \nabla u^h - f \big\|_{H^{-5/2 - \epsilon}(\partial \Omega)}, \ \big\| \nu \cdot \sigma \nabla (\tilde{u}^h \circ \Phi) - f \big\|_{H^{-5/2 - \epsilon}(\partial \Omega)} \, \leq \, C_{\epsilon} h^2 \|I \|_{\R^M},
\end{equation}
where $C_\epsilon > 0$ is independent of $0 < h \leq 1$ (but not of $\Phi$).
Moreover,
\begin{equation}
\label{perush}
\big\|\nu \cdot \sigma \nabla (u^h - \tilde{u}^h\circ \Phi)  \big\|_{H^{-1}(\partial \Omega)} \, \leq \, C  h^{3/2} \| I \|_{\R^M},
\end{equation}
where $C>0$ is also independent $0 < h \leq 1$ (but not of $\Phi$).
\end{lemma}

\begin{proof}
Since both $u^h$ and $\tilde{u}^h\circ \Phi$ satisfy the first, second and fourth conditions of~\eqref{eq:cemeqs}, the first estimate \eqref{Neumann} directly follows from the line of reasoning in the proof of \cite[Lemma~3.2]{Hanke11b}.

In order to deduce \eqref{perush}, let $\varphi \in H^1(\partial \Omega)$ be arbitrary and denote its mean value over $E^h_m$ by $\bar{\varphi}_m^h$, $m = 1, \dots, M$. Moreover, we set $f^h = \nu \cdot \sigma \nabla u^h$ and $\tilde{f}^h = \nu \cdot \sigma \nabla (\tilde{u}^h \circ \Phi)$ on $\partial \Omega$. Mimicking the proof of \cite[Lemma~3.2]{Hanke11b}, we start by writing
\begin{equation*}
\big\langle f^h - \tilde{f}^h, \varphi \big\rangle_{\partial \Omega} 
   \, =\, \sum_{m=1}^M \int_{E_m^h} \big(f^h - I_m/|E_m^h| \big) \varphi \, {\rm d} S 
         +  \sum_{m=1}^M \int_{E_m^h} \big(I_m/|E_m^h| - \tilde{f}^h \big) \varphi \, {\rm d} S,
\end{equation*}
and then estimate all terms appearing on the right-hand side in the same manner. 

Indeed, since $f^h - I_m/|E_m^h|$ has vanishing mean on $E_m^h$, 
\begin{align*}
\int_{E_m^h} \big(f^h - I_m/|E_m^h| \big) \varphi \, {\rm d} S \, &= \, 
\int_{E_m^h} \big(f^h - I_m/|E_m^h| \big) (\varphi - \bar{\varphi}^h_m) \, {\rm d} S \\[1mm]
\, & \leq \, \big\|f^h - I_m/|E_m^h| \big \|_{L^2(E_m^h)} \|\varphi - \bar{\varphi}^h_m  \|_{L^2(E_m^h)}.
\end{align*}
Following the same logic as in the second part of the proof of \cite[Lemma~3.2]{Hanke11b}, one can show that
$$
\big\|f^h - I_m/|E_m^h| \big \|_{L^2(E_m^h)} \, = \, C h^{1/2} \| I \|_{\R^M}.
$$
Moreover, by the Poincar{\'e} inequality for a convex domain~\cite{Payne60},
$$ 
\|\varphi - \bar{\varphi}^h_m  \|_{L^2(E_m^h)} \, \leq \, C h \| \varphi \|_{H^1(E_m^h)} .
$$
Combining the previous three estimates, we obtain that
$$
\int_{E_m^h} \big(f^h - I_m/|E_m^h| \big) \varphi \, {\rm d} S  \, \leq \, C h^{3/2} \| \varphi \|_{H^1(E_m^h)}\| I \|_{\R^M} , \qquad m= 1, \dots, M.
$$
Repeating exactly the same line of reasoning for $\tilde{f}^h$, we also get
$$
\int_{E_m^h} \big(I_m/|E_m^h| - \tilde{f}^h \big) \varphi \, {\rm d} S  \, \leq \, C h^{3/2} \| \varphi \|_{H^1(E_m^h)} \| I \|_{\R^M}, \qquad m= 1, \dots, M,
$$
where the constant this time around depends on $\Phi$.

To sum up, we have altogether demonstrated that
$$
\big\| f^h - \tilde{f}^h \big\|_{H^{-1}(\partial \Omega)} \, = \, \sup_{0 \not= \varphi \in H^1} \frac{\langle f^h - \tilde{f}^h, \varphi \rangle_{\partial \Omega}}{\| \varphi \|_{H^1(\partial \Omega)}} \, \leq \, C h^{3/2}\| I \|_{\R^M},
$$
which completes the proof. \qquad 
\end{proof}

Due to the boundedness from the zero-mean subspace of $H^s(\partial \Omega)$ to $H^{s+1}(\partial \Omega)/\R$, $s\in\R$, of the Neumann-to-Dirichlet map associated to the conductivity equation with a smooth conductivity~\cite{Lions72}, the second estimate of Lemma~\ref{tokalemma} also implies
\begin{equation}
\label{Dirichlet}
\big\| u^h - \tilde{u}^h \circ \Phi \big\|_{L^2(\partial \Omega)/\R} \, \leq \, C h^{3/2} \| I \|_{\R^M}
\end{equation}
for any $0 < h \leq 1$.

The following main theorem of this section demonstrates that $\tilde{U} = U + O(h^{1/2})$ in the topology of $\R^M / \R$ as $h>0$ goes to zero. Take note that such a result cannot be straightforwardly deduced by subtracting the variational formulations of \eqref{eq:cemeqs} and \eqref{eq:cemeqsD}, followed by an obvious change of variables: As hinted by \eqref{eq:bound} and \eqref{Neumann}, the $\mathcal{H}(\Omega)$-norm of $(u^h, U^h)$ does not stay bounded as $h>0$ tends to zero, which reduces the applicability of such a variational argument. In particular, the mean {\em current  densities} through the electrodes explode when the electrodes shrink, leading also to unbounded growth of the electrode voltages due to the potential jumps caused by the contact resistances.

\begin{theorem}
\label{paalause}
For all $0 < h \leq 1$, it holds that
$$
\big\| U^h - \tilde{U}^h \big \|_{\R^M / \R} \, \leq \, C h^{1/2} \|I \|_{\R^M},
$$
where $C(\Omega, \sigma, \{z_m\}, \{y_m\},\Phi) > 0$ is independent of $h$.
\end{theorem}

\begin{proof}
We begin by fixing the representatives of the equivalence classes $U^h, \tilde{U}^h \in \R^M/ \R$ to be the mean-free ones, but abuse the notation by denoting them with the same symbols, i.e., $U^h, \tilde{U}^h \in \R_\diamond^M$. Notice that this also fixes the ground levels for $u^h$ and $\tilde{u}^h$. The corresponding piecewise constant functions on the electrodes of $\partial \Omega$ are
$$
\mathbb{U}^h = \sum_{m=1}^M U_m^h \chi_m^h \in L^2(\partial \Omega)
\qquad  {\rm and} \qquad
\tilde{\mathbb{U}}^h = \sum_{m=1}^M \tilde{U}_m^h \chi_m^h \in L^2(\partial \Omega), 
$$
respectively, with $\chi_m^h$ denoting the characteristic function of $E_m^h$.
It is relatively easy to see that for $0<h\leq 1$,
\begin{equation}
\label{isoU}
\big\| U^h - \tilde{U}^h \big\|_{\R^M} 
\, \leq \, \frac{C}{h} \, \big\| \mathbb{U}^h - \tilde{\mathbb{U}}^h \big\|_{H^{s}(\partial \Omega)/\R}, \qquad s < 1/2,
\end{equation}
where $C = C(s)>0$ can be chosen to be independent of $h$; see Lemma~\ref{apulemma} below. We also introduce the contact resistance functions
$$
Z = \sum_{m=1}^M z_m \chi_m^1 \in C^\infty\big(\overline{E^1}\big)
\qquad  {\rm and} \qquad
\tilde{Z} = \sum_{m=1}^M \frac{\tilde{z}_m}{|\Phi'|} \chi_m^1 \in C^\infty\big(\overline{E^1}\big),
$$
where the motivation for the latter comes from \eqref{uusikolmas}. We extend $Z$ and $\tilde{Z}$ onto the whole boundary as elements of $C^\infty(\partial \Omega)$, denoting them still by the same symbols.
Finally, let $\chi^h = \sum_{m=1}^M\chi^h_m$ be the characteristic function of $E^h$.

Fix $\epsilon >0$. Due to the triangle inequality, the second and third equations of \eqref{eq:cemeqs}, and the Robin condition \eqref{uusikolmas},
\begin{align}
\label{elecest}
 \big\| \mathbb{U}^h - \tilde{\mathbb{U}}^h\big\|_{H^{-5/2-\epsilon}(\partial \Omega)/\R} \,   & \leq \,    
\big \| \chi^h (u^h - \tilde{u}^h\circ \Phi) \big \|_{H^{-5/2-\epsilon}(\partial \Omega)/\R} \\[1mm] 
&  \ \ \quad +   \big \|Z \nu \cdot \sigma \nabla u^h - \tilde{Z} \nu \cdot \sigma \nabla (\tilde{u}^h\circ \Phi) \big\|_{H^{-5/2-\epsilon}(\partial \Omega)}. \nonumber
\end{align}
Let us start with the first term on the right-hand side of \eqref{elecest}:
\begin{align}
\label{ekatermi}
\big \| \chi^h (u^h - \tilde{u}^h\circ \Phi) \big \|_{H^{-5/2-\epsilon}(\partial \Omega)/\R} \, &\leq \, \big \| \chi^h (u^h - \tilde{u}^h\circ \Phi) \big \|_{L^2(\partial \Omega)/\R} \nonumber \\[1mm]
\, &\leq  \, \| u^h - \tilde{u}^h \circ \Phi \|_{L^2(\partial \Omega)/\R}
\, \leq \, C h^{3/2} \| I \|_{\R^M},
\end{align}
where the last inequality is~\eqref{Dirichlet}.

To estimate the second term on the right-hand side of \eqref{elecest}, 
notice first that
$$
Z f \, = \, \tilde{Z} f 
$$
in the sense of distributions on $\partial \Omega$ since $f$ of \eqref{affa} is a linear combination of delta distributions and $Z, \tilde{Z} \in C^\infty(\partial \Omega)$ coincide on the support of $f$, that is,
$$
\tilde{Z}(y_m) =  \frac{\tilde{z}_m}{|\Phi'(y_m)|} = z_m = Z(y_m), \qquad m=1, \dots, M. 
$$
As a consequence, by virtue of the triangle inequality, 
\begin{align}
\label{tokatermi}
\big \|Z \nu \cdot \sigma \nabla u^h - \tilde{Z} \nu \cdot \sigma \nabla (\tilde{u}^h\circ \Phi) \big\|
\, & \leq \, 
\big\| Z \big(\nu \cdot \sigma \nabla u^h - f \big) \big \| \nonumber \\[1mm] 
& \ \ \quad + \big\| \tilde{Z} \big(f -  \nu \cdot \sigma \nabla (\tilde{u}^h\circ \Phi \big)\big) \big\| \nonumber \\[1mm]
& \, \leq \,  C h^2 \| I \|_{\R^M} 
\end{align}
where all norms are those of $H^{-5/2 - \epsilon}(\partial \Omega)$ and the last inequality is an easy consequence of \eqref{Neumann}.

The assertion now follows by combining \eqref{isoU}--\eqref{tokatermi}. \qquad 
\end{proof}

The following lemma complements Theorem~\ref{paalause} and completes this section.

\begin{lemma}
\label{apulemma}
Each $V \in \R_\diamond^M$ and the associated piecewise constant function $\mathbb{V}^h  =  \sum_{m=1}^M V_m \chi_m^h \in L^2(\partial \Omega)$ satisfy
$$
\| V \|_{\R^M} \, \leq \, \frac{C_s}{h} \, \big\| \mathbb{V}^h \big\|_{H^s(\partial \Omega)/\R} \qquad {\rm for} \ {\rm all} \ s \leq -1/2 \ {\rm and} \ 0 < h \leq 1,
$$
with some $C_s(\partial \Omega, \{ y_m \}) > 0$ independent of $h$ and $V$. 
\end{lemma}

\begin{proof}
Fix $s \leq -1/2$ and choose functions $\varphi_l \in H^{-s}(\partial \Omega)$, $l=1, \dots, M$, such that 
$$
\varphi_l \, \equiv \, \frac{1}{|E^1_m|} \delta_{l,m} \quad {\rm on} \ E^1_m \qquad
{\rm and} \qquad \int_{\partial \Omega} \varphi_l \, {\rm d} S = 1 
$$
for $l,m = 1, \dots, M$ and with $\delta_{l,m}$ being the Kronecker delta. In other words, $\varphi_l$ is constant on the $l$th {\em reference} electrode $E_l^{1}$, its support does not intersect $E^1_m$ for $m\not=l$, and its integral over $\partial \Omega$ is independent of $l = 1,\dots, M$. (The precise value of the integral is not important.)

For each $V \in \R_\diamond^M$, we introduce the corresponding test function
$$
\phi_V \, = \, \sum_{m=1}^{M} V_m \varphi_m \in H^{-s}_\diamond(\partial \Omega)
$$
where the mean-free Sobolev space $H^{-s}_\diamond(\partial \Omega)$ is defined as
$$
H^{-s}_\diamond(\partial \Omega) \, = \, \left\{ \psi \in  H^{-s}(\partial \Omega) \ \Big| \ \int_{\partial \Omega} \psi\, {\rm d} S = 0 \right\}.
$$
In particular, observe that $H^{-s}_\diamond(\partial \Omega)$ realizes the dual of $H^s(\partial \Omega)/\R$. Let us define 
$$
\| V \|_s \, := \, \|\phi_V \|_{H^{-s}(\partial \Omega)}, \qquad V \in \R_\diamond^M.
$$
Since $\| \, \cdot \, \|_s$ is obviously a norm, there exists a constant $C_s>0$ such that
\begin{equation}
\label{equinorm}
\| V \|_s \, \leq \,  C_s \| V \|_{\R^M} \qquad \text{for all } V \in \R_\diamond^M
\end{equation}
due to the finite-dimensionality of $\R_\diamond^M$.

It follows from the aforementioned duality between $H^{s}(\partial \Omega)/\R$ and $H^{-s}_\diamond(\partial \Omega)$, together with \eqref{Eh} and \eqref{equinorm}, that
\begin{align*}
\big \| \mathbb{V}^{h} \big \|_{H^s(\partial \Omega) / \R} \, &= \, 
\sup_{0 \not= \phi \in H^{-s}_{\diamond}} \frac{\langle \mathbb{V}^{h}, \phi \rangle_{\partial \Omega}}{\|\phi\|_{H^{-s}(\partial \Omega)}} \, \geq \, \frac{\langle \mathbb{V}^{h}, \phi_V \rangle_{\partial \Omega}}{\|\phi_V\|_{H^{-s}(\partial \Omega)}} \\ 
\, &= \, \frac{1}{\| V \|_s} \sum_{m=1}^M \frac{|E_m^h|}{|E^1_m|} \, V_m^2 \, =  \,h \, \frac{\|V\|_{\R^M}^2}{\| V \|_s} 
\, \geq  \, \frac{h}{C_s} \|V\|_{\R^M}
\end{align*}
for all $0 \not= V \in \R_\diamond^M$. This completes the proof. \qquad
\end{proof}

\begin{remark} Although Lemma~\ref{apulemma} obviously also holds for all $-1/2 < s < 1/2$, these Sobolev indices were excluded as one should expect better estimates for them. As an example, it is easy to check that
$$
\| V \|_{\R^M} \, \leq \, Ch^{-1/2} \big\| \mathbb{V}^h \big\|_{L^2(\partial \Omega)} \qquad {\rm for} \ {\rm all} \ V \in \R^M.
$$
On the other hand, the Sobolev norms corresponding to $s\geq 1/2$ are not finite for $\mathbb{V}^h$ unless $V = 0$.
\end{remark}

\section{Reconstruction algorithm}
\label{sec:inverse}

In this section we formulate an iterative reconstruction algorithm that is a basic adaptation of the Gauss--Newton iteration to the minimization of a Tikhonov-type objective functional arising from the Bayesian inversion paradigm \cite{Kaipio05}. For a closely related implementation, see~\cite{Darde13a}, where the estimation of the object boundary is also included in the algorithm. In what follows, all conductivities are assumed to be isotropic.

Let 
\begin{equation}
\label{data}
{\bf V} = \big[(V^{(1)})^{\rm T},\ldots, (V^{(M-1)})^{\rm T}\big]^\T\in \R^{M(M-1)}
\end{equation}
denote a noisy set of measurements where the electrode voltages $\{V^{(j)}\}_{j=1}^{M-1}\subset \R_\diamond^M$ correspond to a basis of net currents $\{I^{(j)}\}_{j=1}^{M-1}\subset \R_\diamond^M$ injected through the electrodes $\{E_m\}_{m=1}^M$ attached to $\partial\Omega$. Given ${\bf V}$, our aim is to simultaneously estimate the conductivity distribution, the contact resistances, and the electrode locations in a prescribed reconstruction domain $D$, which may (or may not) differ from the true target domain $\Omega$. We assume that the positions and shapes of the electrodes can be parametrized with a finite vector of shape variables denoted by $\tilde{e}$; for explicit examples, see Appendix~\ref{sec:App}. In what follows, the conductivity $\tilde{\sigma}$ in the reconstruction domain $D$ is assumed to have been discretized beforehand. In the numerical examples of Section~\ref{sec:numerics}, this is achieved by resorting to piecewise linear representations in FE bases. In particular, all further references to $\tilde{\sigma}$ are to be understood in the finite-dimensional Euclidean sense.
 
Assuming an additive zero-mean Gaussian noise model and Gaussian priors for the unknowns, determining a MAP estimate for the parameters of interest corresponds~\cite{Darde13a} to finding a minimizer for the objective function
\begin{equation}\label{eq:tikhfun}
F(\tilde{\sigma}, \tilde{z}, \tilde{e}) \, = \, \big\|\tilde{\bf U}(\tilde{\sigma}, \tilde{z},\tilde{e})-{\bf V}\big\|_{\Gamma^{-1}_{0}}^2 + \|\tilde{\sigma}-\tilde{\sigma}^\mu \|_{\Gamma^{-1}_1}^2 +
\|\tilde{z}-\tilde{z}^\mu \|_{\Gamma^{-1}_2}^2  + \|\tilde{e}-\tilde{e}^\mu \|_{\Gamma^{-1}_3}^2
\end{equation}
where we have used the notation $\|x\|_A^2 := x^\T A x$. In \eqref{eq:tikhfun}, $\tilde{\bf U}(\tilde{\sigma}, \tilde{z},\tilde{e}) \in \R^{M(M-1)}$ is the forward solution evaluated in the reconstruction domain $D$ for the (discretized) conductivity $\tilde{\sigma}$, the contact resistances $\tilde{z}$ and the electrode parameters $\tilde{e}$ (cf.~\eqref{data}).
The covariance matrix of the zero-mean Gaussian noise is denoted by $\Gamma_{0}$. In addition, $\tilde{\sigma}^\mu$, $\tilde{z}^\mu$, $\tilde{e}^\mu$ and the symmetric positive definite matrices $\Gamma_1, \Gamma_2$, $\Gamma_3$ are the mean values and the covariances, respectively, of the underlying prior probability distributions for the to-be-estimated variables.

The functional $F$ can be iteratively minimized using an adaptation of the Gauss-Newton algorithm, which requires evaluating the Jacobian matrices of the forward map with respect to the parameters of interest. We refer to~\cite{Kaipio00,Vilhunen02} for a detailed description of the numerical approximation of the Jacobian matrices $J_{\tilde{\sigma}}$ and $J_{\tilde{z}}$ of $\tilde{\mathbf{U}}$ with respect to $\tilde{\sigma}$ and $\tilde{z}$, respectively. On the other hand, the Jacobian matrix $J_{\tilde{e}}$ of $\tilde{\bf U}$ with respect to the electrode parameters is computed with the help of the sampling formula~\eqref{eq:sampling}; the details are given in Section~\ref{subsec:comp}. We introduce the total Jacobian  ${\bf J} = [J_{\tilde{\sigma}}, J_{\tilde{z}}, J_{\tilde{e}}]$ which is a matrix-valued function of the variable triplet $(\tilde{\sigma},\tilde{z},\tilde{e})$. A generic description of the full algorithm is as follows:
\begin{algorithm} \label{alg:1}
Assume the covariance matrices $\Gamma_{0}, \Gamma_1, \Gamma_2, \Gamma_3$ and the expected values $\tilde{z}^\mu, \tilde{e}^\mu$ are given. Compute  Cholesky factorizations $L^\T_i L_i = \Gamma^{-1}_i$ for all $i = {0}, 1, 2, 3$ and build the block-diagonal matrix ${\bf L} = \diag (L_1, L_2, L_3)$. Choose the prior mean of the conductivity to be the homogeneous estimate $\tilde{\sigma}^\mu = \tau_{\rm min}{\bf 1}$ where
$$
\tau_{\rm min} = \underset{\tau \in \R_+}{\arg \min} \,\big| L_{0} \big(\tilde{\bf U} (\tau{\bf 1},\tilde{z}^\mu,\tilde{e}^\mu)- {\bf V}\big) \big|^2,
$$
and ${\bf 1} = [1, \ldots, 1]^\T$ is a constant vector of the appropriate length. Initialize with the compound variable
$$
{\bf b}^{(0)} =
\begin{bmatrix}
\tilde{\sigma}^\mu\\
\tilde{z}^\mu\\
\tilde{e}^\mu
\end{bmatrix}.
$$
While the cost functional $F({\bf b}^{(j)})$ decreases sufficiently, iterate for $j = 0, 1, 2, \ldots$:
\begin{enumerate}
\item Form
\begin{align*}
{\bf A} = \begin{bmatrix}
L_{0} {\bf J}({\bf b}^{(j)}) \\ {\bf L}
\end{bmatrix}
\qquad \text{and} \qquad
{\bf y} = \begin{bmatrix}
L_{0} \big(\tilde{\bf U}({\bf b}^{(j)})-{\bf V}\big) \\ {\bf L} ({\bf b}^{(j)} - {\bf b}^{(0)})
\end{bmatrix}.
\end{align*}
\item Solve the direction $\Delta{\bf b}$ from 
$$
\Delta{\bf b} = \underset{{\bf x}}{\arg\min} \, |{\bf A}{\bf x} - {\bf y}|^2.
$$
\item Set ${\bf b}^{(j+1)} = {\bf b}^{(j)} - q \Delta{\bf b}$, where the step size $q > 0$ is chosen by a line search.
\end{enumerate}
\end{algorithm}

\vspace{2mm}

\begin{remark}\label{rem:init-z}
In the initialization phase of Algorithm~\ref{alg:1}, we could simultaneously compute a homogeneous estimate for the contact resistances and use it as the corresponding prior mean. However, we have instead chosen to employ artificially large, user-specified  prior means in our numerical examples, as high contact resistances yield better numerical stability~\cite{Darde16}. See \cite{Darde13b} for a similar approach.
\end{remark}

\subsection{Computation of the electrode derivative}\label{subsec:comp}
Let $\tilde{U} \in \R^M$ be the potentials at the electrodes on $\partial D$ corresponding to a current pattern $I \in \{I^{(j)}\}_{j=1}^{M-1}\subset \R_\diamond^M$ and some given values for $\tilde{\sigma}$, $\tilde{z}$ and $\tilde{e}$. We consider how to compute the derivative of $\tilde{U}$ with respect to an arbitrary component $\tilde{e}_k$ of $\tilde{e}$. We assume that $\tilde{e}_k$ only affects the shape and/or position of a single electrode, say, $\tilde{E}_m$, and only consider the case $n=3$, which is arguably the more challenging one. 

Without loss of generality, we may assume that the aim is to evaluate the derivative with respect to $\tilde{e}_k \in \R$ at the origin.
Keeping the other components of $\tilde{e}$ fixed, suppose that $\partial\tilde{E}_m = \partial\tilde{E}_m(e_k)$ can be parametrized in an open neighborhood of the origin by $\tilde{\gamma}_m(\xi,\tilde{e}_k)$, with $\xi \in [0,2\pi)$ being a path variable. Assume that $\tilde{\gamma}_m$ is twice continuously differentiable and set
$$
a \big(\tilde{\gamma}_m(\xi,0)\big) \, := \, \frac{\partial \tilde{\gamma}_m}{\partial \tilde{e}_k }(\xi,0), \qquad \xi \in [0,2\pi),
$$
and $a\equiv 0$ on $\partial E_l$ for $l\not= m$. Obviously, for a small enough $\eps >0$,
\begin{equation}\label{eq:ae}
\tilde{\gamma}_m(\xi,\eps) - \tilde{\gamma}_m(\xi,0) \, = \, a \big(\tilde{\gamma}_m(\xi,0)\big) \eps + O(\eps^2),  \qquad \xi \in [0,2\pi).
\end{equation} 
Recalling \eqref{eq:diffquo} and using \eqref{eq:ae}, it is not difficult to deduce that
\begin{equation}\label{eq:dU}
\frac{\partial \tilde{U}}{\partial \tilde{e}_k}\Big|_{\tilde{e}_k=0}  = \, \tilde{U}'(I,0) a.
\end{equation}
By computing the right-hand side of \eqref{eq:dU} for all $I \in \{I^{(j)}\}_{j=1}^{M-1}$ with the help of \eqref{eq:sampling}, one thus obtains the $k$th column of $J_{\tilde{e}}$ (cf.~\eqref{data}).

In the numerical examples of Section~\ref{sec:numerics}, we consider two choices for $D$: a two-dimensional disk and a three-dimensional right circular cylinder. The corresponding parametric formulas that enable the numerical implementation of shape differentiation are documented in Appendix~\ref{sec:App}. However, we emphasize that the global closed-form parametrizability of $\partial D$ and $\partial\tilde{E}$ is not an indispensable requirement. For example, local parametrizations using,~e.g.,~splines provide a flexible computational framework for generalizing the described method to (almost) arbitrary geometries.

\section{Numerical experiments}
\label{sec:numerics}
We consider three numerical examples that investigate the performance of Algorithm~\ref{alg:1} from different perspectives. The leading idea is to test whether a simultaneous reconstruction of various parameters yields better conductivity images than a naive approach that ignores the geometric mismodeling between the reconstruction and target domains. In Example~1, the measurement data is simulated in the unit square whereas the reconstruction is computed in the unit disk, which corresponds to a significant modeling error. Example~2 applies Algorithm~\ref{alg:1} to experimental data from a thorax-shaped water tank containing different configurations of embedded inclusions. Since the target is homogeneous in the vertical direction, the geometry is essentially two-dimensional; the reconstruction is formed in a disk that has the same circumference as the tank.  
To conclude, Example~3 investigates how Algorithm~\ref{alg:1} performs in an inherently three-dimensional geometry. 

All computations are performed using a {\em finite element method} (FEM) with $\mathbb{P}_1$ (piecewise linear) basis functions \cite{Kaipio00, Vauhkonen97, Vilhunen02}. In particular, both the conductivity and the internal electric potential $u$ are discretized in terms of $\mathbb{P}_1$. Because the model for the measurement geometry changes at each iteration step during a simultaneous reconstruction of $\tilde{\sigma}$, $\tilde{z}$ and $\tilde{e}$, which also induces changes in the FE meshes, the conductivity values are stored in a static `reference mesh' via interpolation.

\subsection*{{\it Example~1: Simulations in a square, reconstructions in a disk}}\label{ex:1}
In this example, the target domain $\Omega$ is the unit square $[-1,1]\times[-1,1]$ and the reconstruction domain $D$ is the unit disk. There are three electrodes of width $0.25$ attached to each side of the square, i.e., there are altogether $M=12$ electrodes (cf.~Figure~\ref{sfig:t1-target}). The target conductivity illustrated in Figure~\ref{sfig:t1-target} consists of two inclusions lying in a homogeneous background. As a reference, we also show the target configuration mapped onto the unit disk by the conformal $\Phi\colon \Omega \to  D$ that keeps the coordinate axes fixed; take note that this obviously is not the only possible choice for $\Phi$ and also recall that we denote $\Psi = \Phi^{-1}$. The synthetic contact resistances are drawn form a normal distribution,~$z_m^{\rm trgt} \sim \mathcal{N}(0.1,0.01^2)$, and the transformed ones are calculated as $\tilde{z}_m^{\rm trgt} = |\Phi' (y_m)|\, z_m^{\rm trgt}$, $m=1, \dots, M$,
where $y_m$ is the center point of the $m$th electrode on $\partial \Omega$. To simulate the data, $\Omega$ is discretized into a triangular mesh consisting of $3.3 \cdot 10^4$ nodes and $6.4 \cdot 10^4$ elements with suitable refinements around the electrodes.

The potential measurements ${\bf V} \in \R^{M(M-1)}$ are synthesized by corrupting the FEM-approximated electrode potentials ${\bf U} \in \R^{M(M-1)}$ with additive zero-mean Gaussian noise with the covariance matrix
\begin{align}\label{noise-cov}
\Gamma_{0} = \Big( \eta_{0} \max_{i,j} |{\bf U}_i - {\bf U}_j| \Big)^2 \! \id \, \in \R^{M(M-1) \times M(M-1)},
\end{align}
where 
$\eta_{0} = 10^{-3}$. The prior covariance matrix for the conductivity consists of the entries (cf.,~e.g.,~\cite{Lieberman2010})
\begin{align}\label{prior}
\Gamma_1^{(ij)} = \eta_1^2 \exp \left(-\frac{|x_i-x_j|^2}{2\lambda^2} \right) 
\end{align}
where the pointwise variance is $\eta_1^2 = 0.5$, the correlation length $\lambda = 1$, and $x_i, x_j \in D \subset \R^2$ are the coordinates of the mesh nodes corresponding to the coefficients $\tilde{\sigma}_i, \tilde{\sigma}_j \in \R_+$, respectively.
The prior mean and the covariance matrix for the contact resistances are 
$\tilde{z}^\mu = {\bf 1} \in \R^M$ and $\Gamma_2 = \eta_2^2 \,\id \in \R^{M \times M}$ with $\eta_2=10$, respectively; see Remark~\ref{rem:init-z}.

The results of our first test are documented in Figure~\ref{sfig:t1-cmapel}.
The electrodes are fixed to the `correctly mapped' positions $\tilde{E}_m = \Phi(E_m)$, and only the conductivity and the contact resistances are reconstructed. Algorithm~\ref{alg:1} initializes with the homogeneous estimate 
$\tilde{\sigma}^\mu = 0.87$. Figure~\ref{sfig:t1-cmapel} shows the final reconstruction in the disk $D$ and as a $\Psi$-mapped version in the original (but in practice unknown) square domain~$\Omega$. In addition, the reconstructed contact resistances are compared with the target values~$\tilde{z}^{\rm trgt}$. According to Figure~\ref{sfig:t1-cmapel}, it is possible to form a reasonable reconstruction of the conductivity in the mismodeled domain $D$ by using the conformally mapped electrodes. Moreover, the reconstructed contact resistances seem to mimic $\tilde{z}^{\rm trgt}$ as predicted by Theorem~\ref{paalause}.

\begin{figure}
\begin{center}
\subfloat[Left: the target conductivity and electrodes in the original domain $\Omega$. Middle: the target conductivity and electrodes conformally mapped onto the reconstruction domain $D$. Right: the contact resistances $z^{\rm trgt}$ (filled dots) and the transformed ones $\tilde{z}^{\rm trgt}$ (hollow dots).]{\label{sfig:t1-target}
\includegraphics[width=0.29\textwidth]{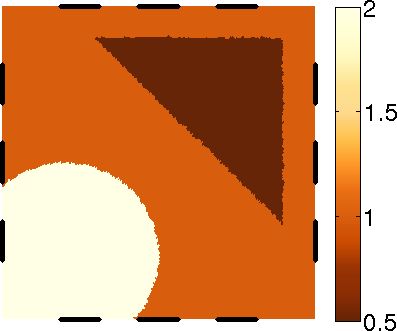} \quad
\includegraphics[width=0.29\textwidth]{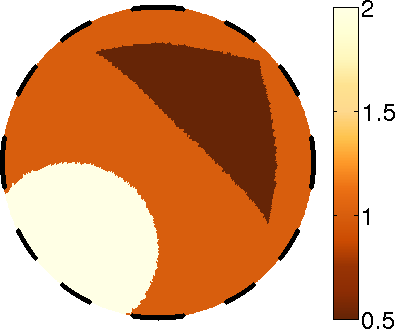}
\quad
\includegraphics[width=0.34\textwidth]{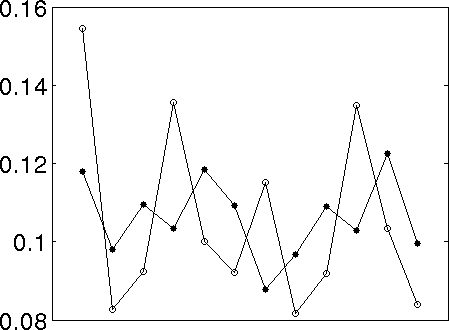}}\\[4mm]
\subfloat[Reconstruction with conformally mapped electrodes.]{\label{sfig:t1-cmapel}
\includegraphics[width=0.29\textwidth]{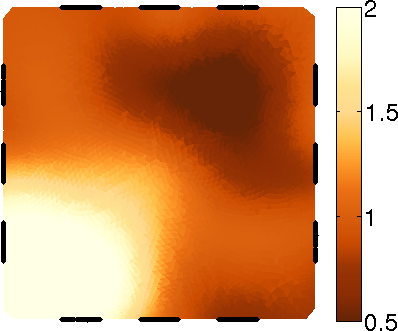} \quad
\includegraphics[width=0.29\textwidth]{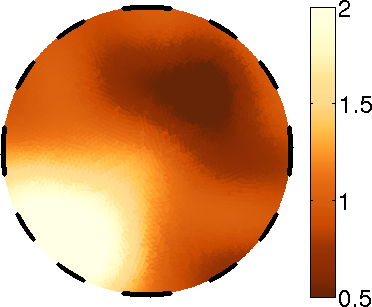}
\quad
\includegraphics[width=0.34\textwidth]{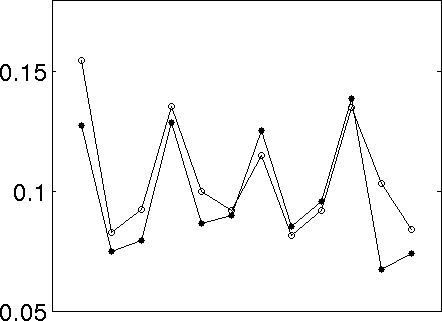}}\\[4mm]
\subfloat[Reconstruction with fixed equally spaced electrodes.]{\label{sfig:t1-fixel}
\includegraphics[width=0.29\textwidth]{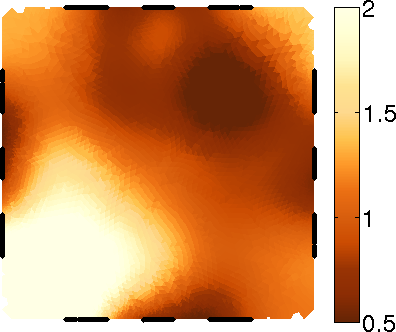} \quad
\includegraphics[width=0.29\textwidth]{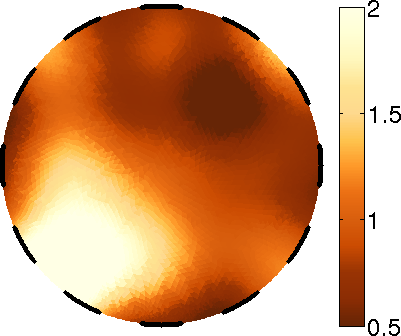}
\quad
\includegraphics[width=0.34\textwidth]{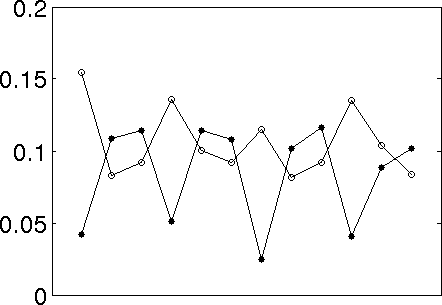}}\\[4mm]
\subfloat[Simultaneous reconstruction of all parameters.]{\label{sfig:t1-freeel}
\includegraphics[width=0.29\textwidth]{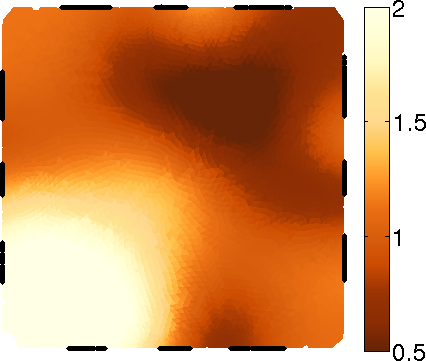} \quad
\includegraphics[width=0.29\textwidth]{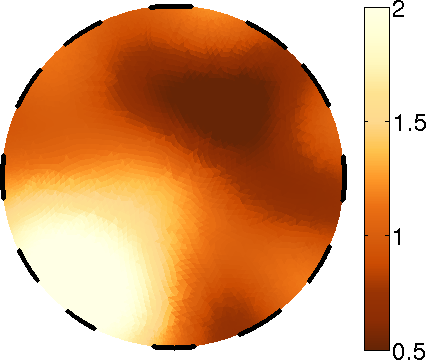}
\quad
\includegraphics[width=0.34\textwidth]{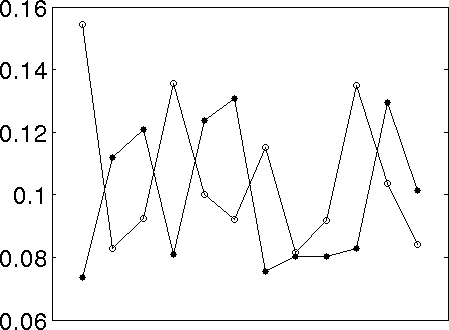}}\\[4mm]
\subfloat[
Reconstruction with contact resistances fixed to $\tilde{z}^{\rm trgt}$.
]{\label{sfig:t1-fixz}
\includegraphics[width=0.29\textwidth]{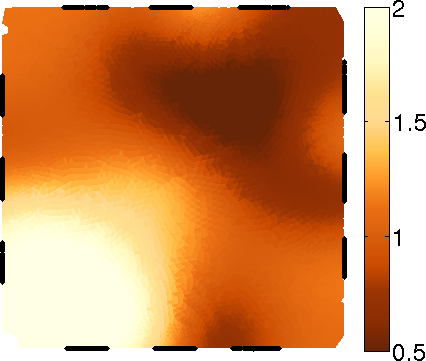} \quad
\includegraphics[width=0.29\textwidth]{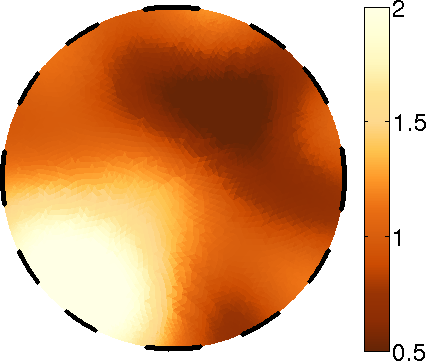} \quad
\includegraphics[width=0.34\textwidth]{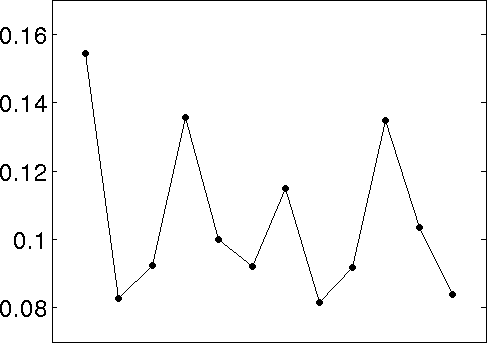}}
\caption{\label{fig:test1} 
{\it Example~1}. All reconstructions are computed in the unit disk (middle column) and conformally mapped to the original domain (left column). Apart from the top right image, the reconstructed contact resistances are plotted with filled dots and $\tilde{z}^{\rm trgt}$ with hollow dots (right column). The values correspond to the electrodes in counter-clockwise order starting from the rightmost one.
}
\end{center}
\end{figure}

The second test considers a more realistic set-up where the correctly mapped electrodes are not utilized. However, we still use fixed electrodes $\tilde{E}_m(\theta^\mu_m, \alpha^\mu_m)$, $m=1, \dots, M$, with the central angles $\theta^\mu \in \R^M$ equally spaced around $\partial D$ and the radii/half-widths $\alpha^\mu_m = 0.125$, $m=1, \dots, M$, in accordance with the size of the true electrodes on $\partial \Omega$; see Appendix~\ref{sec:App} for the details about the parametrization of the electrodes. The algorithm starts at the homogeneous estimate 
$\tilde{\sigma}^\mu = 0.95$ and the output reconstruction is presented in Figure~\ref{sfig:t1-fixel}. The effect of the target inclusions is still visible, but a lot of information is lost compared to the reconstructions in Figure~\ref{sfig:t1-cmapel}. Moreover, the estimated contact resistances do not resemble $\tilde{z}^{\rm trgt}$.

Next, we employ the full version of Algorithm~\ref{alg:1}, i.e., we also reconstruct the central angles $\theta$ and radii $\alpha$ of the electrodes. The electrode shape variable $\tilde{e}$ defined in~\eqref{eq:e1} is given the prior covariance $\Gamma_3 = \eta_3^2 \, \id \in \R^{2M \times 2M}$, where $\eta_3=0.125$. As the prior mean $\tilde{e}^\mu$ we use the parameter values from the previous test, that is, equally-spaced electrodes of radius $0.125$. The results are presented in Figure~\ref{sfig:t1-freeel}. The reconstruction of the conductivity is improved compared to the previous test; in fact, it is comparable to the one with correctly mapped electrodes presented in Figure~\ref{sfig:t1-cmapel}. However, the reconstructed electrodes do not coincide with the conformally mapped ones and the reconstructed contact resistances still do not agree with $\tilde{z}^{\rm trgt}$. Our hypothesis is that the effects of increasing the electrode widths and decreasing the contact resistances are so similar that the algorithm does not distinguish them: 
On `too large' electrodes, instead of decreasing their size, the algorithm reconstructs higher contact resistances than the corresponding components of $\tilde{z}^{\rm trgt}$. Similarly we get low contact resistances for `too small' electrodes.

In order to verify the aforementioned hypothesis, we conclude this first example by considering the simultaneous reconstruction of the conductivity and the electrode parameters when the contact resistances are (unrealistically) fixed to the target values $\tilde{z}^{\rm trgt}$. The results are illustrated in Figure~\ref{sfig:t1-fixz}. In this case, the reconstructed electrodes end up near the conformally mapped ones and, when mapped to $\partial \Omega$ by $\Psi$, they almost coincide with the original electrodes. However, we also observe that fixing the contact resistances to the values $\tilde{z}^{\rm trgt}$ has practically no effect on the reconstruction of the conductivity in comparison to Figure~\ref{sfig:t1-freeel}.

The computations related to the reconstructions of this example were performed on FE meshes having around $8\cdot 10^3$ nodes and $1.5 \cdot 10^4$ elements with appropriate refinements near the electrodes. In the final two examples where the electrodes were allowed to move, a new FE mesh was created at each iteration of Algorithm~\ref{alg:1}. In these cases, the intermediate conductivity reconstructions were stored in a static reference mesh with around $10^4$ nodes and $2 \cdot 10^4$ triangles. The conformal maps $\Phi$ and $\Psi$ were constructed with the Schwarz--Christoffel Toolbox for MATLAB~\cite{Driscoll96}.

\subsection*{{\it Example~2: Two-dimensional real-life data}}
This test applies Algorithm~\ref{alg:1} to experimental data measured on a thorax-shaped cylindrical water tank with cross-sectional circumference of $106$\,cm. There are $M = 16$ rectangular metallic electrodes of width $2$\,cm attached to the internal lateral surface of the tank. The electrode height is $5$\,cm, which equals the water depth as well as the height of the inclusions placed inside the tank. In particular, as the target is homogeneous in the vertical direction and no current flows through the top or the bottom of the water layer, one can model the measurements by the two-dimensional CEM; see,~e.g.,~\cite{Gehre12} for more details. Three target configurations with either one or two embedded inclusions are considered; see~Figure~\ref{fig:test2}. The measurements were performed with the {\em Kuopio impedance tomography} (KIT4) device at the University of Eastern Finland using low-frequency ($1$\,kHz) alternating current~\cite{Kourunen09}. The phase information of the measurements is not used, but the amplitudes of currents and potentials are interpreted as real numbers. In what follows, the units of distance, conductivity and contact resistance are cm, mS/cm and k$\Omega\, {\rm cm}^2$, respectively.

\begin{figure}
\begin{center}
\subfloat[One hollow steel cylinder with a rectangular cross section.]{\label{sfig:t2-case1}
\includegraphics[width=0.3\textwidth]{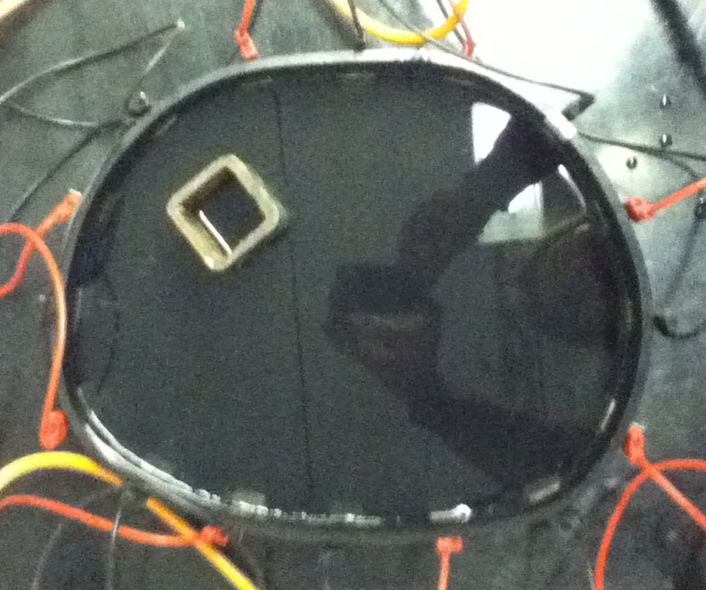} \quad
\includegraphics[width=0.3\textwidth]{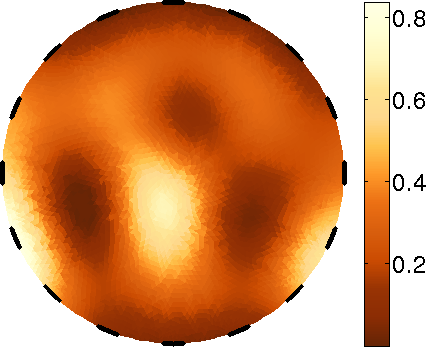}
\quad
\includegraphics[width=0.3\textwidth]{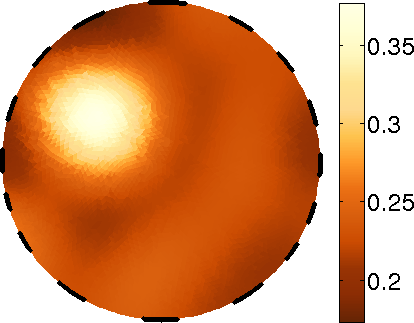}}\\[4mm]
\subfloat[One hollow steel cylinder with a rectangular cross section and one plastic cylinder with a round cross section.]{\label{sfig:t2-case2}
\includegraphics[width=0.3\textwidth]{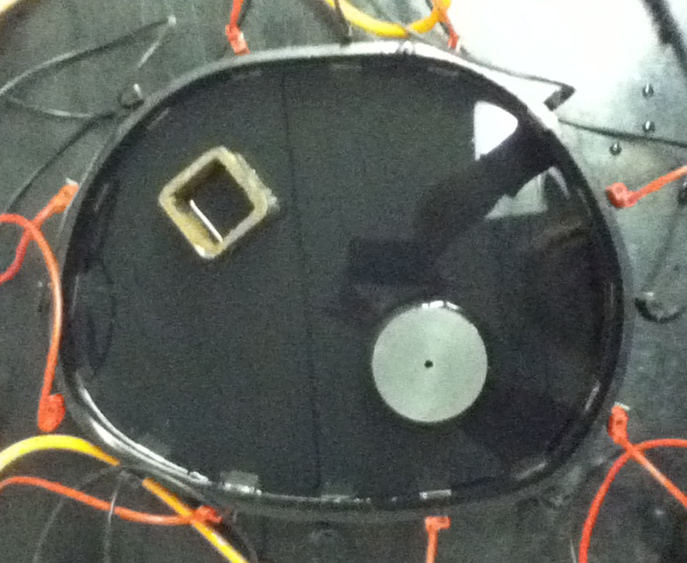} \quad
\includegraphics[width=0.3\textwidth]{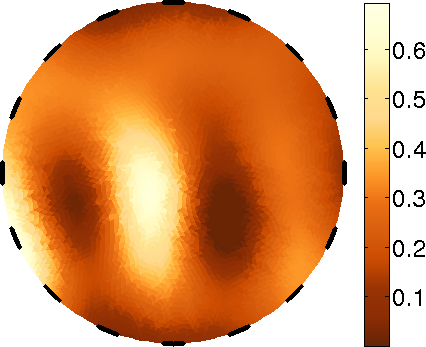}
\quad
\includegraphics[width=0.3\textwidth]{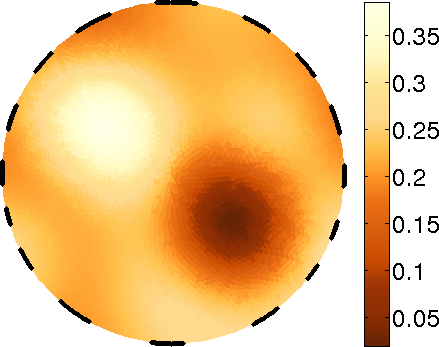}}\\[4mm]
\subfloat[One plastic cylinder with a round cross section and one hollow steel cylinder with a rectangular cross section.]{\label{sfig:t2-case3}
\includegraphics[width=0.3\textwidth]{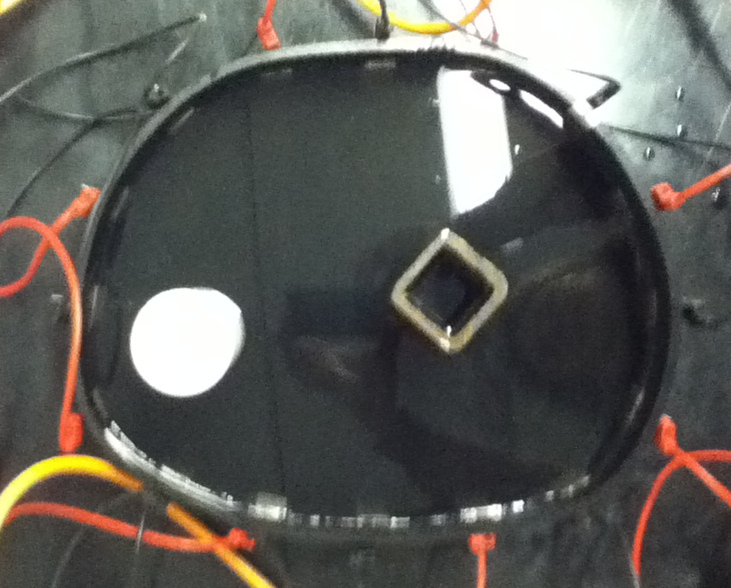} \quad
\includegraphics[width=0.3\textwidth]{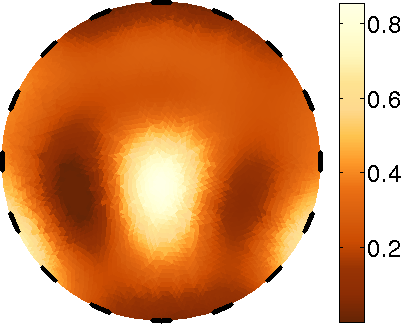}
\quad
\includegraphics[width=0.3\textwidth]{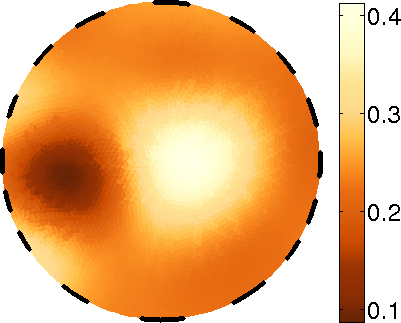}}
\caption{\label{fig:test2}
{\it Example~2}. Left column: the measurement configurations. Middle column: the reconstructions with fixed equally spaced electrodes of the correct width. Right column: the reconstructions produced by the full Algorithm~\ref{alg:1}. The unit of conductivity is mS/cm.
}
\end{center}
\end{figure}

The reconstruction algorithm is run in the origin-centered disk $D = B(0,r)$ of radius $r = 106/(2\pi)$. The noise covariance $\Gamma_{0}$ is of the form \eqref{noise-cov} with $\eta_{0} = 10^{-2}$ and ${\bf U}$ replaced by the measured potential data ${\bf V}$. The prior covariance matrices are of the same form as in Example~1; the parameter values are chosen to be $\lambda = r$  and $\eta_1^2 = 0.5$ in \eqref{prior}, and $\eta_2 = 10$ and $\eta_3 = 1$ for the covariances of $\tilde{z}$ and $\tilde{e}$, respectively. The corresponding mean values also mimic those in Example~1: 
$\tilde{z}^\mu = {\bf 1} \in \R^M$, the central angles of the electrodes $\theta^\mu \in \R^M$ are equally spaced, and the
angular electrode radii are $\alpha^\mu = 1/r \cdot {\bf 1} \in \R^M$.

For each target configuration shown in the left-hand column of Figure~\ref{fig:test2}, we run the reconstruction algorithm with two different presets: first with the fixed electrodes $\tilde{E}_m(\theta^\mu_m,\alpha^\mu_m)$, $m=1, \dots, M$, and then employing the full Algorithm~\ref{alg:1},~i.e.,~simultaneously reconstructing the conductivity, the contact resistances and the electrode parameters. The results are presented in Figure~\ref{fig:test2}. With fixed electrodes, the conductivity reconstruction is poor. On the other hand, when the electrode locations and widths are estimated as a part of the algorithm, the reconstruction of the conductivity is free from significant artifacts and it accurately reproduces the qualitative properties of the target, such as the number and the approximate locations of the inclusions. We stress that the reconstructions are qualitatively comparable to those computed in an accurately modeled domain~\cite{Darde13b}.

The initial homogeneous estimates for the conductivity were 
$\tilde{\sigma}^\mu = 0.23$, $\tilde{\sigma}^\mu = 0.21$ and $\tilde{\sigma}^\mu = 0.24$ in the configurations of Figure~\ref{sfig:t2-case1}, Figure~\ref{sfig:t2-case2} and Figure~\ref{sfig:t2-case3}, respectively. The computations related to the reconstructions of this example were performed on FE meshes having around $7\cdot 10^3$ nodes and $1.3 \cdot 10^4$ triangles with appropriate refinements near the electrodes. When the estimation of the electrode parameters was included in the algorithm, the intermediate conductivity reconstructions were stored on a static reference mesh with around $1.1 \cdot 10^4$ nodes and $2.1 \cdot 10^4$ triangles.

\subsection*{{\it Example~3: Three-dimensional cylinder}}

In the final numerical experiment, we investigate how the introduced reconstruction technique performs in three spatial dimensions. We start with a simple `nearly two-dimensional' test where the target conductivity distribution is vertically homogeneous. The domain $\Omega$ is a cylinder $\Omega_0 \times (0,h)$ with height $h = 0.5$ and $\partial\Omega_0$ parametrized by 
$$
\gamma(\theta) \, = \, \left( \frac{3}{\sqrt{1.5^2 \cos^2 \theta + 2^2 \sin^2 \theta}} + 0.75e^{-(\theta-\pi)^6} + 0.6 \cos \theta \sin(-2\theta) \right) \begin{bmatrix}
\cos \theta \\ \sin \theta
\end{bmatrix}
$$
where $\theta \in [0,2\pi)$. On the lateral surface $\partial\Omega_0\times (0,h)$ there are $M = 16$ circular electrodes of radius $0.15$ centered at height $h/2$ and angular positions $2\pi(m-1)/M$, $m=1, \dots, M$. The simulation domain $\Omega$ and the target conductivity are visualized in the leftmost column of Figure~\ref{fig:test3a}. The target contact resistances are drawn from $\mathcal{N}(0.1,0.01^2)$ as in Example~1.

\begin{figure}
\begin{center}
\includegraphics[width=0.3\textwidth]{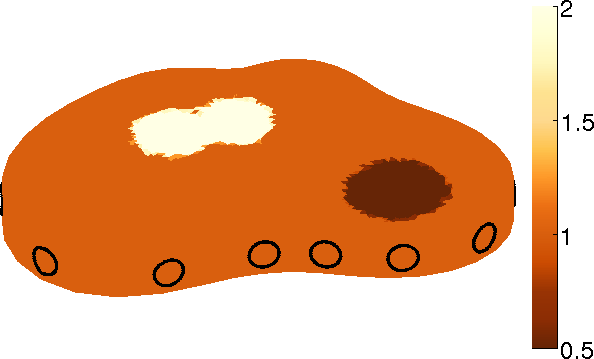} \quad
\includegraphics[width=0.3\textwidth]{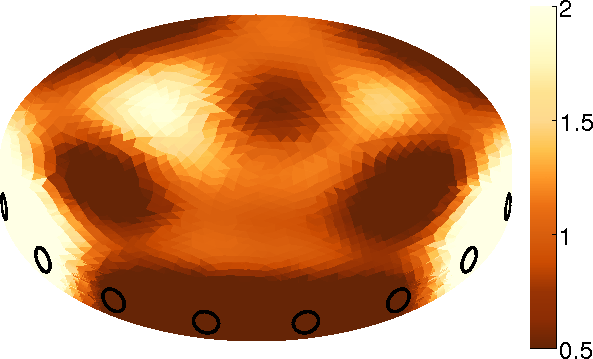}
\quad
\includegraphics[width=0.3\textwidth]{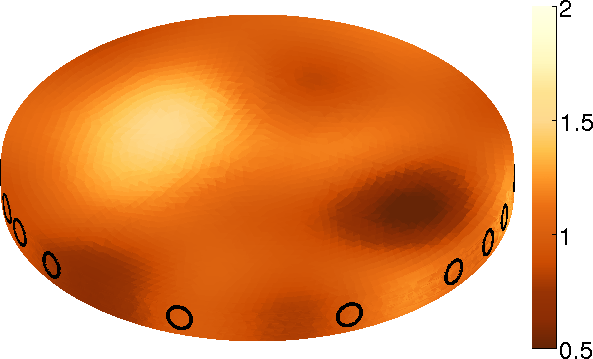}\\[4mm]
$\vcenter{\hbox{\includegraphics[width=0.3\textwidth]{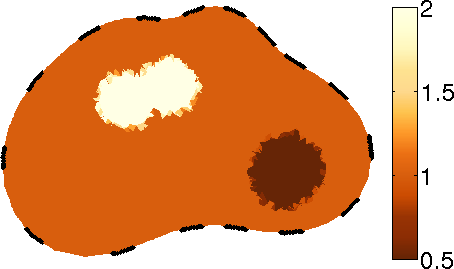}}}$ \quad
$\vcenter{\hbox{\includegraphics[width=0.3\textwidth]{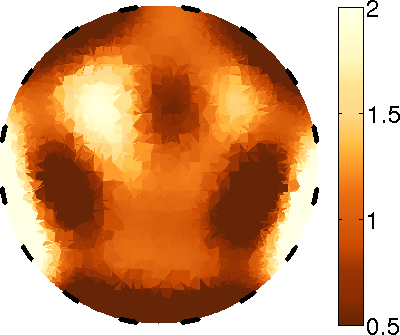}}}$
\quad
$\vcenter{\hbox{\includegraphics[width=0.3\textwidth]{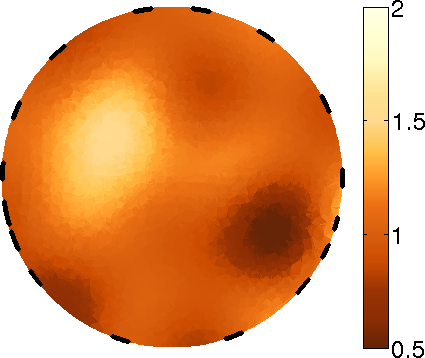}}}$
\caption{\label{fig:test3a}
{\it Example~3, one electrode belt}. Left column: the measurement configuration. Middle column: the reconstruction with fixed equally spaced electrodes of the correct shape. Right column: the reconstruction produced by the full Algorithm 1. Top row: three-dimensional images. Bottom row: slices at height $h/2$.
}
\end{center}
\end{figure}

The measurements ${\bf V} \in \R^{M(M-1)}$ are once again synthesized by corrupting the FEM-approximated electrode potentials $U^{(j)}$, $j=1, \dots, M-1$,
by additive Gaussian noise. To be more precise, this time around $V_m^{(j)} = U_m^{(j)} + W_m^{(j)}$, with $W_m^{(j)} \sim \mathcal{N}(0,\eps_m^{(j)})$ and
\begin{equation}\label{noise-var}
\eps_m^{(j)} \, = \, 0.01^2 |U^{(j)}_m|^2 + 0.001^2 \max_{1 \leq n,p \leq M} |U_n^{(j)}-U_p^{(j)}|^2,
\end{equation}
where $m=1, \dots, M$ and $j=1,\dots,M-1$. 
The reconstruction algorithm is run in the domain $D$ which is a right circular cylinder with height $h = 0.5$ and radius $r = 3$. We use the target electrode parameters as the corresponding prior means,~i.e.,~$\theta^\mu_m = 2\pi(m-1)/M$, $\zeta^\mu_m = h/2$ and $\ell^\mu_m = k^\mu_m = 0.15$ for $m=1, \dots, M$. For more information on the parametrization of the geometry, see Appendix~\ref{sec:App}. Moreover, 
$\tilde{z}^\mu = {\bf 1} \in \R^M$. The prior covariances are selected as in Examples~1 and 2, that is, $\Gamma_1$ is defined by \eqref{prior} with $\lambda = r$ and $\eta_1^2 = 0.5$, $\Gamma_2 = \eta_2^2\, \id \in \R^{M \times M}$ and $\Gamma_3 = \eta_3^2 \, \id \in \R^{4M \times 4M}$ with $\eta_2 = 10$ and $\eta_3 = 0.15$.

As in Example~2, we perform a visual comparison between a reconstruction computed with the electrode parameters fixed to $\theta^\mu$, $\zeta^\mu$, $\ell^\mu$, $k^\mu$ and a reconstruction produced by Algorithm~\ref{alg:1} in its complete form. The algorithm initializes with the homogeneous conductivity 
$\tilde{\sigma}^\mu = 0.98$. The final reconstructions are presented in Figure~\ref{fig:test3a}. The simultaneous retrieval of the electrode parameters clearly improves the reconstruction when compared to the fixed-electrode approach. 

\begin{figure}
\begin{center}
\includegraphics[width=0.3\textwidth]{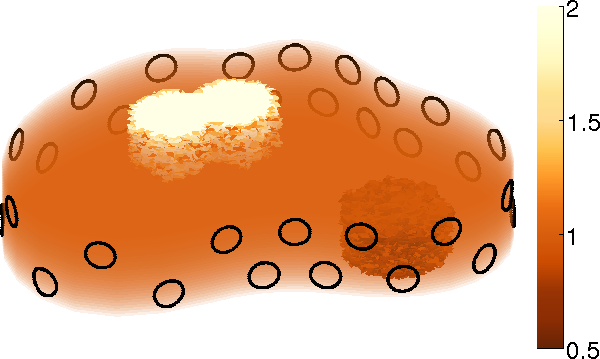} \quad
\includegraphics[width=0.3\textwidth]{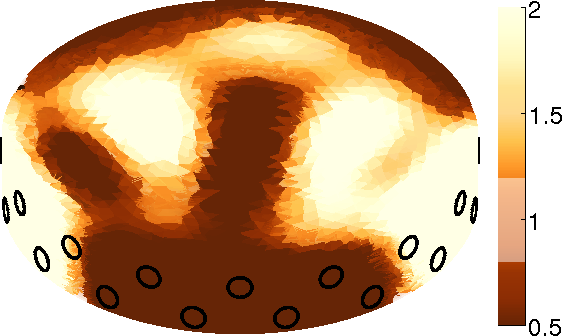}
\quad
\includegraphics[width=0.3\textwidth]{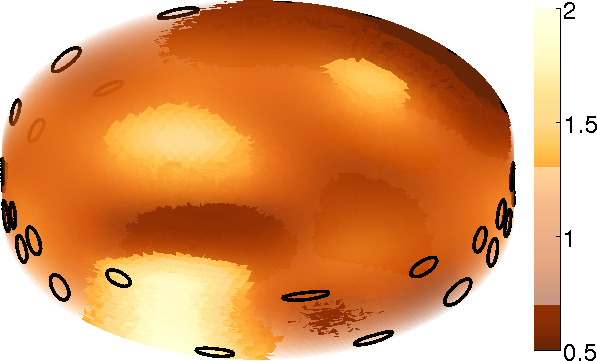}\\[4mm]
$\vcenter{\hbox{\includegraphics[width=0.25\textwidth]{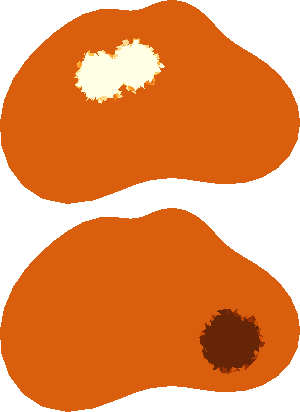}}}$ \qquad \quad
$\vcenter{\hbox{\includegraphics[width=0.25\textwidth]{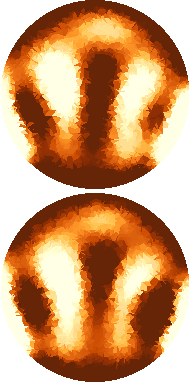}}}$
\qquad \quad
$\vcenter{\hbox{\includegraphics[width=0.25\textwidth]{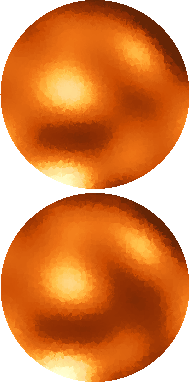}}}$ \qquad 
\caption{\label{fig:test3b}
{\em Example~3, two electrode belts}. Left column: the measurement configuration. Middle column: the reconstruction with fixed equally spaced electrodes of the correct shape. Right column: the reconstruction produced by the full Algorithm 1. Top row: three-dimensional images (the values indicated in the colorbars are transparent). Middle row: slices at height $3h/4$. Bottom row: slices at height $h/4$.
}
\end{center}
\end{figure}

In the second test of this example, we consider a measurement configuration that is genuinely three-dimensional. The target domain $\Omega$ is as in the previous test apart from its height that is increased to $h = 1$. Moreover, this time there are two belts of sixteen electrodes with radii $0.15$ assembled at heights $h/4$ and $3h/4$, respectively ($M=32$). The central angles of the electrodes are $[4\pi(m-1)/M]_{m=1}^{M/2}$ and $[4\pi(m-1)/M + 2\pi/M]_{m=1}^{M/2}$ in the lower and the upper electrode belt, respectively. In addition, the target conductivity is vertically inhomogeneous as the heights of the two inclusions are only $h/2$, with the conductive one touching the top and the insulating one the bottom of $\Omega$; see the left column of Figure~\ref{fig:test3b}.

We choose the reconstruction domain $D$ to be the right circular cylinder with height $h = 1$ and radius $r = 3$. The noise model is the same as in the previous test,~i.e.,~it is in accordance with \eqref{noise-var}. The prior means of the electrode shape variables are set to the corresponding values in the true measurement configuration for~$\Omega$; the prior covariances are as in the first test of this example. In this setting, the algorithm starts with the homogeneous estimate 
$\tilde{\sigma}^\mu = 0.99$. The final reconstructions for the fixed-electrode case and with the simultaneous reconstruction of all parameters of interest are shown in the middle and the right column of Figure~\ref{fig:test3b}, respectively. As in Examples~1 and~2, including the estimation of the electrode parameters in the algorithm improves the reconstruction, but in this three-dimensional setting the increase in quality is less obvious: the images in the right-hand column of Figure~\ref{fig:test3b} also suffer from significant artifacts. In particular, regions of too high or too low conductivity emerge close to those boundary sections  where the shapes of $\partial \Omega$ and $\partial D$ differ the most. On the positive side, some traits of the vertical inhomogeneity in the target conductivity are also present in the reconstruction.

The reconstructed electrodes for the two tests of this example are visualized in Figure~\ref{fig:test3els}. The gaps between the electrodes seem to correlate with the local geometric modeling errors; see also the slices in the bottom row of Figure~\ref{fig:test3a}. In the case of two electrode belts, the reconstructed electrode heights also vary although the geometric mismodeling is only related to the cross section of the cylinder. It is difficult to find any intuitive patterns in the reconstructed electrode shapes. 
According to our experience, the quality of the reconstruction is also affected by the height of the cylindrical domain: if one chooses $h=1$ in the first test, the conductivity reconstruction produced by the full Algorithm~\ref{alg:1} is clearly worse than the one in Figure~\ref{fig:test3a}.

In the first test, the data was simulated on a FE mesh with $2.6 \cdot 10^4$ nodes and $1.2 \cdot 10^5$ tetrahedra. The reconstruction meshes had around $7 \cdot 10^3$ nodes and $2.5 \cdot 10^4$ tetrahedra, and the reference mesh for storing the conductivity had $1.4 \cdot 10^4$ nodes and $6.5 \cdot 10^4$ tetrahedra.
In the second test, the simulation mesh had $4.6 \cdot 10^4$ nodes and $2.2 \cdot 10^5$ tetrahedra. The reconstruction meshes consisted of about $1.3 \cdot 10^4$ nodes and $5.5 \cdot 10^4$ tetrahedra, and the reference mesh of $1.8 \cdot 10^4$ nodes and $9.5 \cdot 10^4$ tetrahedra.

\begin{figure}
\centering
$\vcenter{\hbox{\includegraphics[width=0.3\textwidth]{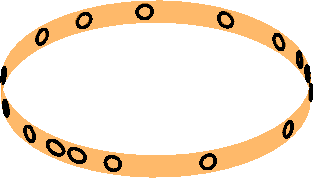}}}$ \quad
$\vcenter{\hbox{\includegraphics[width=0.3\textwidth]{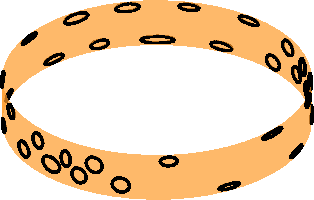}}}$
\caption{\label{fig:test3els}
{\it Example~3}. Reconstructed positions and shapes of the electrodes in Example~3; see the right-hand columns of Figures~\ref{fig:test3a} and \ref{fig:test3b}.
}
\end{figure}

\section{Conclusion}
\label{sec:conclusion}
We have demonstrated --- both numerically and theoretically --- that one can recover from mismodeling of the object shape in two-dimensional EIT by allowing electrode movement in an output least squares reconstruction algorithm. Although the same conclusion does not apply to three spatial dimensions, in simple cylindrical settings estimating the positions and shapes of the electrodes as a part of the reconstruction algorithm seems to alleviate the artifacts caused by an inaccurate model for the object boundary.

We have also tested the introduced algorithm in a setting where the reconstruction domain $D$ is a ball and the target domain $\Omega$ is a slightly distorted ball. The corresponding numerical results are not presented here, but they are in line with those in the second test of Example~3: including the estimation of the electrode shapes and positions in an output least squares algorithm does not significantly improve conductivity reconstructions in inherently three-dimensional settings.

\section*{Acknowledgments}
We would like to thank Professor Jari Kaipio's research group at the University of Eastern Finland (Kuopio) for granting us access to their EIT devices and Professor Antti Hannukainen for letting us use his finite element solver.

\appendix
\section{Explicit formulas for shape derivatives}\label{sec:App}

\subsection{Disk}

Let $ D$ be the two-dimensional origin-centered disk of radius $r > 0$, meaning that each $\tilde{E}_m$ is an open arc segment. The boundary circle is parametrized by $\tilde{\gamma}(\theta) = r[\cos\theta, \sin\theta]^\T$, $\theta \in [0,2\pi)$. We choose the electrode shape variables to be the central angles and angular half-widths/radii of the electrodes,
\begin{equation}
\label{eq:e1}
\tilde{e} \, = \, [\theta_1,\ldots,\theta_M,\alpha_1,\ldots \alpha_M]^\T \in \R^{2M}.
\end{equation}
In particular, the end points, i.e., the `boundaries', of the $m$th electrode are given as $\{x_m^-, x_m^+\} = \{\tilde{\gamma}(\theta_m - \alpha_m), \tilde{\gamma}(\theta_m + \alpha_m)\}$. A straightforward calculation gives 
\begin{equation}\label{eq:disk}
\nu_{\partial\tilde{E}}(x_m^{\pm}) \cdot \frac{\partial x_m^\pm}{\partial \theta_m} \, = \, \pm r, \qquad \nu_{\partial\tilde{E}}(x_m^{\pm}) \cdot \frac{\partial x_m^\pm}{\partial \alpha_m} \, = \, r.
\end{equation}
Since $\partial \tilde{E}_m$ consists of two points, the integrals in \eqref{eq:sampling} reduce to two-point evaluations involving \eqref{eq:disk}.

\subsection{Right circular cylinder with ellipsoidal electrodes}

Let $D$ be a right circular cylinder with radius $r > 0$ and height $h > 0$. We assume that $\partial\tilde{E}_m$, $m=1, \dots, M$, is an ellipse attached (without stretching) to the lateral boundary of $D$ so that one of the two semiaxes is parallel to the axis of $D$. By a tedious but straightforward calculation we find a parametrization $\tilde{\gamma}_m\colon [0,2\pi) \mapsto \partial\tilde{E}_m$, 
\[
\tilde{\gamma}_m(\xi) \, = \ r\sin\bigg(\frac{\ell_m}{r}\cos\xi\bigg)
\hspace{-3pt}
\begin{bmatrix}
-\sin\theta_m \\ \cos\theta_m \\ 0
\end{bmatrix}
\hspace{-3pt} 
+ r\cos\bigg( \frac{\ell_m}{r} \cos\xi \bigg)
\hspace{-3pt}
\begin{bmatrix}
\cos\theta_m \\ \sin\theta_m \\ 0
\end{bmatrix}
\hspace{-3pt}
+ (k_m\sin\xi + \zeta_m) 
\hspace{-3pt}
\begin{bmatrix}
0 \\ 0 \\ 1
\end{bmatrix}
\]
where $\ell_m, k_m > 0$ are the semiaxes of the $m$th electrode ellipse, and the center of mass of $\tilde{E}_m$ projected onto $\partial D$ is given by $[r\cos\theta_m, r\sin\theta_m, \zeta_m]^\T$, $\theta_m\in [0,2\pi)$, $\zeta_m\in(0,h)$. The relevant shape parameter vector is 
\begin{equation}\label{eq:e2}
\tilde{e} \, = \, [\theta_1,\ldots,\theta_M,\zeta_1,\ldots,\zeta_M,\ell_1,\ldots,\ell_M,k_1,\ldots,k_M]^\T \in \R^{4M}.
\end{equation}
After some basic calculations, we end up with
\begin{equation}\label{eq:cylinder}
\left| \tilde{\gamma}'_m (\xi) \right| \left(\nu_{\partial\tilde{E}} \big( \tilde{\gamma}_m (\xi) \big) \cdot \frac{\partial \tilde{\gamma}_m}{\partial \omega} (\xi) \right) = \,
\begin{cases}
r k_m \cos \xi, & \omega = \theta_m,\\
\ell_m \sin \xi, & \omega = \zeta_m,\\
k_m\cos^2\xi, & \omega = \ell_m,\\
\ell_m\sin^2\xi, & \omega = k_m,
\end{cases} 
\end{equation}
which can be used to evaluate the curve integrals in \eqref{eq:sampling}. 

\bibliographystyle{acm}
\bibliography{cmeem-refs}

\end{document}